\documentclass[12pt,a4paper,draft]{amsart}
\usepackage{version}
\usepackage{amsmath, amssymb,amsfonts,amsthm}
\usepackage{esint}
\usepackage{amsmath,amsthm,amssymb,latexsym,epsfig,graphicx,subfigure}
\numberwithin{equation}{section}

\newtheorem{thm}{Theorem}[section]

\newtheorem{lm}[thm]{Lemma}

\theoremstyle{definition}

\theoremstyle{definition}
\newtheorem{rem}[thm]{Remark}

\newcommand{\R}{\mathbb{R}}

\newcommand{\C}{\mathbb{C}}
\newcommand{\F}{\mathcal{F}}
\newcommand{\D}{\mathcal{D}}
\renewcommand{\a}{\alpha}
\newcommand{\e}{\varepsilon}

\begin{document}
\title[Existence of principal values]{Existence of principal values of some singular integrals on Cantor sets, and Hausdorff dimension}
\author{J. Cuf\'i, J. J. Donaire, P. Mattila and J. Verdera}
\maketitle

\begin{abstract}
Consider a standard Cantor set in the plane of Hausdorff dimension $1.$ If the linear density  of the associated measure $\mu$ vanishes, then the set of points where the principal value of the Cauchy singular integral of $\mu$ exists has Hausdorff dimension $1.$  The result is extended to Cantor sets in $\R^d$ of Hausdorff dimension $\alpha$ and Riesz singular integrals of homogeneity $-\alpha,$ $0<\alpha<d:$
the set of points where the principal value of the Riesz singular integral of $\mu$ exists has Hausdorff dimension $\alpha.$ A martingale associated with the singular integral is introduced to support the proof.

\bigskip


\medskip

\end{abstract}

\section{Introduction}

Our main result deals with the Cauchy singular integral on Cantor sets in the plane and the proof extends with minor variations to the Riesz transforms in $\R^d.$ We first proceed to formulate the result for the Cauchy integral and then we take care of the Riesz transforms.

The appropriate Cantor sets for the Cauchy integral are defined as follows. Let $(\lambda_n)_{n=1}^\infty$ a sequence of  real numbers satisfying $ \frac{1}{4}\le \lambda_n \le \lambda <\frac{1}{2}.$ Let  $Q_0:= [0,1]\times [0,1]$ be the unit square. Take the $4$ squares contained in $Q_0$ with sides of length $\lambda_1$ parallel to the coordinate axis having a vertex in common with $Q_0$ (the $4$ ``corner squares" of side length $\lambda_1$). Repeat in each of these $4$ squares the same procedure with the dilation factor $\lambda_1$ replaced by $\lambda_2$ to get $16$ squares of side length $\lambda_1 \lambda_2.$ Proceeding inductively we obtain at the $n$-th step $4^n$ squares $Q_j^n,\; 1 \le j \le 4^n,$ of side length $s_n= \lambda_1 \cdots \lambda_n.$ Our Cantor set is
$$K = \bigcap_{n=1}^\infty\bigcup_{j=1}^{4^n}Q_j^n.$$
 Let $\mu$ be the Borel probability measure on $K$  with $\mu(Q_j^n) = 4^{-n}$ and denote by $a_n$ the linear density at generation $n,$
 that is,
 $$a_n = \frac{1}{4^{n}s_{n}} = \frac{\mu(Q_j^n)}{s_{n}} \leq 1.$$ 
Set $\D_n=\{Q_j^n: j=1,\dots,4^n\}$ and $\D=\cup_{n=1}^{\infty}\D_n$.
\vspace{0.1cm}

We then have

\begin{thm}\label{pvthm}
If $\lim_{n\to\infty}a_n=0$, then the set of points $z\in K$ for which the principal value
\begin{equation}\label{pv}
\lim_{\e\to 0}\int_{|w-z|>\e}\frac{1}{w-z}\,d\mu w 
\end{equation}
exists has Hausdorff dimension greater than or equal to $1$.
\end{thm}

This solves a problem posed in \cite[Open problem 5.5, p.1621]{CPV}.

If $a_n=1$ for all $n$, then $K$ is the famous Garnett-Ivanov Cantor set, which has positive and finite one-dimensional Hausdorff measure  but zero analytic capacity. In this case it was noticed in \cite{CPV} that the principal value does not exist at any point of $K$. If $a_n \to 0,$ then the Hausdorff dimension of $K$ is greater than or equal to $1$ and it has non-sigma finite one-dimensional Hausdorff measure. 
If in addition $\sum_na_n^2<\infty,$  then the principal value exists $\mu$ almost everywhere. So Theorem \ref{pvthm} is relevant only when $a_n \to 0$ slowly. 
That the condition $\sum_na_n^2<\infty$ implies the almost everywhere existence of principal values can be seen in two ways. First, we introduce a martingale $(S_n)_{n=0}^\infty$ (see \eqref{m} below) and show that the increments $|S_{n+1}(x) - S_{n}(x)|$ are bounded by $C\,a_n,$ with the constant $C$ independent of $n$ and $x.$ In Lemma \ref{pv1} we prove that for any point $x$ the principal value exists at $x$ if and only $(S_n(x))_{n=0}^\infty$ converges. If $\sum_na_n^2<\infty$, then $S_n$ is an $L^2$ martingale and consequently it converges almost everywhere. Alternatively, the condition $\sum_na_n^2<\infty$ implies that the Cauchy singular integral operator is bounded in $L^2(\mu)$. In \cite{MV} it was shown in a very general setting that $L^2$ boundedness together with zero density of the measure yields the almost everywhere existence of principal values.

The main argument in the proof of Theorem \ref{pvthm} deals with case where $\sum_na_n^2=\infty$.  It is a variation of a line of reasoning used in other situations (see \cite{DLN} and the references there).  We use a stopping time argument to show that  $(S_n(x))_{n=0}^\infty$ converges to 0 in a set of Hausdorff dimension 1 (indeed, given any complex number $z_0$ the martingale $(S_n(x))_{n=0}^\infty$ converges to $z_0$ in a set of Hausdorff dimension 1). We get the dimension $1$ conclusion by applying a lemma of Hungerford \cite{H}. For the sake of the reader we present a proof of Hungerford's lemma  in our context in section \ref{A1}.

Our proof extends with only technical modifications to cover the case of other odd kernels, for instance,
$$\frac{\overline{z}^m}{{z}^{m+1}}, \quad m=1,2,  \dots$$
But one of the ingredients of our method fails for the odd kernel $\frac{z+\overline{z}}{z^2}$ and we do not know whether Theorem \ref{pvthm} holds in this case. The difficulty is indicated at the fifth line after the statement of Lemma \ref{sector}.

In $\R^d$ our proof works for the Riesz transforms of any homogeneity $-\alpha,\; 0< \alpha< d.$ These are the vector valued singular integrals with kernel 
$$R^\alpha(x) = \frac{x}{|x|^{1+\alpha}}, \quad 0 < \alpha <d.$$
The appropriate Cantor sets for the $\alpha$-Riesz transform are those of Hausdorff dimension $\alpha.$  They are constructed by the procedure outlined before in the planar case with dilation factors that satisfy
$2^{-\frac{d}{\alpha}} \le \lambda_n \le \lambda < 2^{-1}.$ At generation $n$ one has $2^{dn}$ cubes $Q_j^n$ of side length $s_n=\lambda_1 \cdots \lambda_n.$ The Cantor set is defined by 
$$K = \bigcap_{n=1}^\infty\bigcup_{j=1}^{2^{dn}} Q_j^n$$
and the canonical measure on $K$ by $\mu(Q_j^n)=2^{-dn}, \; 1 \le j\le 2^{dn}.$
The $\alpha$ density is  $a_n= 2^{-dn} s_n^{-\alpha}=\mu(Q_j^n) s_n^{-\alpha}\le 1.$  For $\lambda_n= 2^{-\frac{d}{\alpha}} ,\; n=1,2 \dots,$
one gets the self similar Cantor set of dimension $\alpha.$ If $a_n \rightarrow 0$ then our Cantor set has Hausdorff dimension $\geq \alpha$ and non $\sigma$ finite  Hausdorff  $\alpha$-dimensional measure. One has

\begin{thm}\label{pvthm2}
If $\lim_{n\to\infty}a_n=0$, then the set of points $x\in K$ for which the principal value
\begin{equation}\label{pvr}
\lim_{\e\to 0}\int_{|y-x|>\e} R^\alpha (y-x)\,d\mu y 
\end{equation}
exists has Hausdorff dimension greater than or equal to $\alpha.$
\end{thm}

In section \ref{A2} we give some indications on how to adapt the proof for the Cauchy kernel to the Riesz transforms in higher dimensions.

We let $\operatorname{diam}(A)$ denote the diameter and $\dim A$ the Hausdorff dimension of a set $A$. We use the notation $a\lesssim b$ to mean that $a \leq C\,b$ for some constant $C$ which may depend on $\lambda$ and $d$, and $a\sim b$ for $a\lesssim b$ and $b\lesssim a$.

\section{Martingales}\label{ma}

Let $C$ be the Cauchy kernel, $C(x) = 1/x$ for $x\in\C, x\neq 0.$  For each $x\in K$ let $Q_n(x)$ be the square in $\D_n$ containing $x.$ Define the truncated Cauchy Integral at generation $n$ as
$$T_{n}(x) = \int_{K\setminus Q_n(x)}C(x-y)\,d\mu y,\quad x\in K,$$
 and a martingale $(S_n(x))_{n=0}^\infty$ by
\begin{equation}\label{m}
S_{n}(x)=S_{Q_n(x)}= \fint_{Q_n(x)}T_{n}\,d\mu, \quad x\in K.
\end{equation}  

\begin{rem}
That $S_n$ is a martingale is easily checked. The reader will realise that the martingale condition also holds for kernels $K(x,y)$ satisfying the antisymmetry condition $K(x,y)=-K(y,x).$
\end{rem}

We shall prove 

\begin{thm}\label{martthm}
If $\lim_{n\to\infty}a_n=0$, then the set of points $x\in K$ for which $(S_n(x))_{n=0}^\infty$ converges has Hausdorff dimension 
greater than or equal to $1$.
\end{thm}

We first show that the martingale \eqref{m} has uniformly bounded increments.

\begin{lm}\label{inc} There exists a positive constant $C=C(\lambda)$ such that
\begin{equation}\label{eq2}
|S_{n+1}(x) - S_{n}(x)|\ \leq Ca_n, \quad n=0,1, \dots \quad\text{and}  \quad x\in K.
\end{equation} 
\end{lm}

Thus if $\sum_na_n$ converges, $(S_{n}(x))_{n=0}^\infty$ converges for all $x\in K$. As mentioned in the introduction, even the weaker condition $\sum_na_n^2<\infty$ implies that $(S_{n}(x))_{n=0}^\infty$ converges for $\mu$ almost all $x\in K$. Hence we shall assume that $\sum_n a_n^2=\infty$. Under this assumption one proves in  \cite{CPV}  that the set where the principal values fail to exist has full $\mu$ measure.  In Lemma \ref{pv1} below we show that principal values exist if and only if the martingale converges. Hence
 $(S_{n}(x))_{n=0}^\infty$ is not convergent for $\mu$ almost all $x \in K.$  By a standard result in martingale theory (see, for example, \cite[Corollary 6, p.561]{S}) we get

\begin{equation}\label{eq2bis}
\limsup_{n\to\infty}|S_{n}(x)-S_m(x)| =\infty, \quad \text{for all}\quad m=0, 1, \dots \quad\text{and }\quad \mu \;\text{a.e.}
\end{equation}
\begin{proof}[Proof of Lemma \ref{inc}]

Set $Q_n=Q_n(x), \; x\in K, \; n=1,2,\dots$  Then

\begin{align*}
&S_{n+1}(x) - S_{n}(x) \\ 
&=\fint_{Q_{n+1}}\int_{K\setminus Q_{n+1}}C(z-y)\,d\mu y\,d\mu z - \fint_{Q_{n}}\int_{K\setminus Q_{n}}C(w-y)\,d\mu y\,d\mu w\\[5pt]
&=\int_{K\setminus Q_{n}} \left(\fint_{Q_{n+1}}C(z-y)\,d\mu z - \fint_{Q_{n}}C(w-y)\,d\mu w \right)\,d\mu y 
\\[5pt] &+ \int_{Q_n\setminus Q_{n+1}}\fint_{Q_{n+1}}C(z-y)\,d\mu z\,d\mu y.\\
\end{align*}
The last double integral is $\lesssim a_n$, where the implicit constant depends on $\lambda$ here and for the rest of the proof.

To estimate the first summand above we remark that for each $z' \in Q_{n+1}$ and $w' \in Q_n$ we have
\begin{align*}
& \fint_{Q_{n+1}}C(z-y)\,d\mu z - \fint_{Q_{n}}C(w-y)\,d\mu w \\*[7pt] 
 &=\fint_{Q_{n+1}}\big(C(z-y)-C(z'-y)\big)\,d\mu z - \fint_{Q_{n}}\big(C(w-y)-C(w'-y) \big)\,d\mu w \\*[7pt] 
 &+ C(z'-y) - C(w'-y).
\end{align*}
Clearly
$$
\left|C(z'-y) - C(w'-y)\right| \lesssim  \,s_n \, |x-y|^{-2}, \quad y \in K\setminus Q_n, \quad x \in Q_n.
$$
Hence
$$
\left|\fint_{Q_{n+1}}\big(C(z-y)-C(z'-y)\big)\,d\mu z\right| \lesssim  \,s_n \, |x-y|^{-2} 
$$
and
$$
\left|\fint_{Q_{n}}\big(C(w-y)-C(w'-y)\big)\,d\mu z\right| \lesssim  \,s_n \, |x-y|^{-2}.
$$
\vspace{0.2cm}

Setting $R_j = Q_j \setminus Q_{j+1}$,  the absolute value of the first summand of $S_{n+1}(x) - S_{n}(x)$ is 

\begin{align*}&\lesssim s_{n}\int_{K\setminus Q_{n}}|x-y|^{-2}\,d\mu y \sim s_{n}\sum_{j=0}^{n-1}s_j^{-2}\mu(R_j)\\ 
&= s_{n}\sum_{j=0}^{n-1}s_j^{-2}4^{-j} \lesssim s_{n}s_{n}^{-2}4^{-n} = a_n,
\end{align*}
because $s_j^{-2}4^{-j} \le (s_{j+1}^{-2} \lambda^2)4^{-j} = (4 \lambda^2)s_{j+1}^{-2}4^{-j-1}$, so $s_j^{-2}4^{-j} \lesssim (4 \lambda^2)^{n-j}s_{n}^{-2}4^{-n}$.
Hence $|S_{n+1}(x) - S_{n}(x)|\ \lesssim a_n.$
\end{proof}

By the following lemma Theorem \ref{martthm} is equivalent with Theorem \ref{pvthm}.

\begin{lm}\label{pv1}
 If $\lim_{n\to\infty}a_n=0,$ then for each $x\in K$ the principal value \eqref{pv} exists if and only if the sequence $(S_n(x))_{n=0}^\infty$ converges.
\end{lm}

\begin{proof} Set $Q_n=Q_n(x)$ for $x\in K$ and $n=1,2, \dots$
Then by the proof of Lemma \ref{inc}
\begin{align*}
\left|S_n(x)-\int_{K\setminus Q_n} \frac{1}{x-y}\,d\mu y \right| &= \left|\fint_{Q_n} \int_{K\setminus Q_n} \left( \frac{1}{x'-y} -\frac{1}{x-y}\right)\,d\mu y\, d\mu x' \right| \\*[7pt]
 &\leq C\,a_n,
\end{align*}
where the constant depends on $\lambda.$
Compare now a given  truncation $\int_{K\setminus B(x,\e)} \frac{1}{x-y}\,d\mu y,$  $0< \e<1,$ with  $\int_{K\setminus Q_n} \frac{1}{x-y}\,d\mu y $ where $n$ is chosen so that $\operatorname{diam}(Q_n)\leq \e < \operatorname{diam}(Q_{n-1}).$ Since $Q_n \subset B(x,\e)$ we have

\begin{align*}
\left|\int_{K\setminus Q_n} \frac{1}{x-y}\,d\mu y - \int_{K\setminus B(x,\e)} \frac{1}{x-y}\,d\mu y\right| &= \left|\int_{B(x,\e) \setminus Q_n} \frac{1}{x-y} \,d\mu y \right| \\*[7pt]
 &\leq C\, \frac{\mu B(x,\e)}{s_n},
\end{align*}
with $C=C(\lambda).$
To complete the proof just remark that, since $\e < \operatorname{diam}(Q_{n-1})$,  $B(x,\e)$ can intersect at most $N$ squares in $\D_n,$ with $N$ an absolute constant. Hence $\mu B(x,\e) \leq C\,\mu(Q_n).$
\end{proof}

\vspace{0.3cm}

We proceed now to discuss relative martingales.

For $x\in R \subset Q, \,Q\in \mathcal D_m,\, R\in \mathcal D_n, \, m<n,$ we define the relative martingale starting at $Q$ as
$$S_{Q,R}(x)=S_{Q,R} = \fint_R \int_{Q\setminus R}C(z-y)\,d\mu y \,d\mu z.$$

Then for some  constant $C$, 
\begin{equation}\label{eq3}
|S_{R} - S_Q - S_{Q,R}| \leq C\,a_m.
\end{equation} 

Indeed, we have
\begin{align*}
&S_{R} - S_Q =\\ 
&\fint_{R}\int_{K\setminus R}C(z-y)\,d\mu y\,d\mu z - \fint_{Q}\int_{K\setminus Q}C(w-y)\,d\mu y\,d\mu w=\\
&\int_{K\setminus Q}(\fint_{R}C(z-y)\,d\mu z - \fint_{Q}C(w-y)\,d\mu w)\,d\mu y + \int_{Q\setminus R}\fint_{R}C(z-y)\,d\mu z\,d\mu y=\\
&\int_{K\setminus Q}\left(\fint_{R}C(z-y)\,d\mu z - \fint_{Q}C(w-y)\,d\mu w\right)\,d\mu y + S_{Q,R}.
\end{align*}
The first summand above is bounded in absolute value by a constant times $a_m$ by the same argument as in the proof of \eqref{eq2}.

As for \eqref{eq2} we have for $R\subset \tilde R \subset Q, Q\in \mathcal D_m, \tilde R\in \mathcal D_n, R	\in \mathcal D_{n+1},$ 

\begin{equation}\label{eq7} 
|S_{Q,R} - S_{Q,\tilde R}| \leq C\,a_n.
\end{equation}

\section{The stopping time argument}\label{st}

The proof of Theorem \ref{martthm} is based on a stopping time argument for which we need some preliminary facts.

Given a non-zero complex number $z$ consider the sector $\sigma(z,\theta), \; 0<\theta <\pi,$ with vertex at $z$ and aperture $\theta$ whose axis is the semi-line emanating from $z$ and passing through $0.$  That is, $w\in \sigma(z, \theta)$ if and only if 
$$
\langle \frac{w-z}{|w-z|}, \frac{-z}{|z|} \rangle \geq \cos(\frac{\theta}{2})
$$
where $\langle \cdot, \cdot \rangle$ denotes the scalar product in the plane. 

The octants with vertex $0$ are the eight sectors
$$\sigma_j = \{w \in \C: w=|w| e^{i \phi},  \; (j-1)\frac{\pi}{4}\le  \phi \le j \frac{\pi}{4} \}, \quad 1\le j \le 8.$$
These are the sectors with vertex the origin of amplitude $45^{\circ}$ degrees and having an edge over a coordinate axis.
It will be convenient to expand these octants so that they have the same axis and amplitude of $75^{\circ}.$ In other words, we are adding $15^{\circ}$ in each direction. Denote the expanded sectors by $\tilde{\sigma}_j.$
The octants with vertex $z$ are the sectors $\sigma_j(z)= z+\sigma_j, \; 1\le j \le 8,$ and the expanded octants $\tilde{\sigma}_j(z)=z+\tilde{\sigma}_j.$

We have the following obvious lemma.

\begin{lm}\label{sector}
Given any sector $\sigma$ of vertex $z$ and amplitude $120^{\circ}$ there exists an octant with vertex $z,$ say $\sigma_j(z)$ for some index $j$ between $1$ and $8,$ such that $\tilde{\sigma}_j(z) \subset \sigma.$
\end{lm}

Consider the symmetries with respect to the coordinate axis and the main diagonal.  That is,
 $f_{1}(x+iy)=-x+iy, f_{2}(x+iy)=x-iy$ and $f_{3}(x+iy)=y+ix$ for $x+iy\in\C$. For any $j,k=1,\dots,8$, by composing two such symmetries we obtain a linear mapping $f_{j,k}$ that maps the octant $\sigma_j$ onto the octant $\sigma_k$. Observe that $C(f_{j}(z))=f_{j}(C(z))$ for $j=1,2$, and  $C(f_{3}(z))=-f_{3}(C(z))$. It is precisely this last identity that fails for the kernel $(z+\overline{z})/z^2.$
 
 Let $Q\in\D$ and let $c_Q$ be its center. Define 
$$f_{Q,j,k}(x)=f_{j,k}(x-c_Q)+c_Q \quad  x\in Q, \quad j,k=1,\dots,8,$$
so that
$$
f_{Q,j,k}(x)-f_{Q,j,k}(y) = f_{j,k}(x-y), \quad x,y \in Q, \quad j,k=1,\dots,8.
$$
We claim that
\begin{equation}\label{eq24}
S_{Q,f_{Q,j,k}(R)}=\e_{j,k}\,f_{j,k}(S_{Q,R}), \quad R\subset Q, \quad Q,R\in\D,
\end{equation}
where $\e_{j,k}=\pm 1.$   We check \eqref{eq24} by the general formula for the image (push-forward) $\nu^{\sharp,f}$ of a measure $\nu$ under a Borel map $f$ (see, for example, \cite[Theorem 1.19]{M})
$$\int_{f(A)}g\,d\nu^{\sharp,f} = \int_{A}(g\circ f)\,d\nu.$$
The restriction of $\mu$ to $Q$ is invariant under the maps $f_{Q,j,k}$, that is, $(\mu|Q)^{\sharp, f_{Q,j,k}} = \mu|Q$. Hence, since $Q\setminus f_{Q,j,k}(R)=f_{Q,j,k}(Q\setminus R)$ and  
$$C\big(f_{Q,j,k}(z)-f_{Q,j,k}(w)\big)=\e_{j,k}\,f_{j,k}(C(z-w)),$$
we obtain
$$\int_{Q\setminus f_{Q,j,k}(R)}\int_{f_{Q,j,k}(R)}C(z-w)\,d\mu z\,d\mu w = \e_{j,k}\,f_{j,k}\left(\int_{Q\setminus R}\int_{R}C(z-w)\,d\mu z\,d\mu w\right),$$
from which \eqref{eq24} follows.

Assume that we have fixed an octant $\sigma_j$ and that for some square $R \in \mathcal{D}$ contained in $Q$ we have $S_{Q,R} \in \sigma_k$ with $k\neq j.$  We claim that we can find a square $R' \in \mathcal{D}$ contained in $Q,$ of the same size as $R,$ such that 
$|S_{Q,R'}|=|S_{Q,R}|$ and $S_{Q,R'} \in \sigma_k.$

If $\e_{j,k}=1$ then the value of the relative martingale at the square $f_{Q,j,k}(R)$ is $f_{j,k}(S_{Q,R}) \in \sigma_j.$  Note that the size of $f_{Q,j,k}(R)$ is exactly the size of $R$ and $\left| S_{Q,f_{Q,j,k}(R)} \right| = \left| S_{Q,R}\right|.$

To treat the case $\e_{j,k}=-1$ let us introduce the mapping $\gamma : Q \rightarrow Q$ defined by $\gamma(x)=-(x-c_Q)+c_Q.$ Then
$\gamma^2$ is the identity mapping on $Q$ and $S_{Q,\gamma(R)}=-S_{Q,R}$ for each square $R\in \mathcal{D}$ contained in $Q.$ Setting $R'=(\gamma \circ f_{Q,j,k})(R)$ we get
\begin{align*}
f_{jk}(S_{Q,R})&=-S_{Q, f_{Q,j,k}(R)}=S_{Q, (\gamma \circ f_{Q,j,k})(R)}=S_{Q,R'}.
\end{align*}

We shall need the following elementary lemma.

\begin{lm}\label{martlemma}
If $z\in\C, w\in\sigma(z, 120^\circ)$ and $0<|w-z|<|z|/2$, then $|w|\leq |z|-|w-z|/4$.
\end{lm}

\begin{proof}
Let $R=|z|, r=|w-z|$ and let $v$ be the third vertex, in addition to 0 and $z$, of the equilateral triangle containing $w$. Under the assumptions of the lemma $|w|$ is maximized when $w$ lies on the side connecting $z$ and $v$. Assuming that $w$ is on that side, project $w$ on the side connecting $0$ and $z$ and apply Pythagoras to obtain
$$|w|^2 = (R-r/2)^2+(\sqrt{3}r/2)^2=r^2+R^2-rR\leq (R-r/4)^2=(|z|-|w-z|/4)^2$$
because of the assumption $r<R/2$.
\end{proof}


\begin{proof}[Proof of Theorem \ref{martthm}]

We  assume, as we may, that $\sum_na_n^2=\infty$. Then for $\mu$ almost all $x$ the sequence $(S_{n}(x))_{n=0}^\infty$ diverges and \eqref{eq2bis} holds.


Let $M$ be a big positive integer to be chosen later. We replace $(a_n)_{n=0}^\infty$ by the non-increasing sequence $b_n=C\max_{m\geq n}a_m$, where $C$ is as in inequalities \eqref{eq2}, \eqref{eq3} and \eqref{eq7}, which now read
\begin{equation}\label{eq2b}\
|S_{n+1}(x) - S_{n}(x)|\ \leq b_n, \quad n=0,1,\dots \quad\text{and}  \quad x\in K,
\end{equation} 
\begin{equation}\label{eq3b}
|S_{R} - S_Q - S_{Q,R}| \leq b_m, \quad \,Q\in \mathcal D_m,\, R\in \mathcal D_n, \,  R \subset Q, 
\end{equation} 
and
\begin{equation}\label{eq7b} 
|S_{Q,R} - S_{Q,\tilde R}| \leq b_n, \quad  \,Q\in \mathcal D_m,\, R\in \mathcal D_{n+1}, \,  \tilde{R}\in \mathcal D_{n}, R \subset \tilde{R}
\subset Q, 
\end{equation}

We plan to define a sequence of stopping time conditions. At each step a family of stopping time squares will arise, which is going to be   the family $\F_{n}$ in Lemma \ref{hungerford} (Hungerford's lemma). The first stopping time is special and its goal is to have a family of squares with relatively large $|S_Q|$ for each square $Q$ in the family.

The first stopping time condition is 
\begin{equation}\label{st1} 
\left| S_Q \right| > M\,b_0.
\end{equation}
Declare $Q$ a stopping time square of first generation if $Q$ is a square in $\D$ for which $\left| S_Q \right| > M\,b_0$ and $\left| S_{Q'} \right| \le M\,b_0, \; Q \subsetneq Q'.$  We call $\F_1$ the set of stopping time squares of first generation. One may think at this as a process as follows. One takes a point $x\in K$ and looks at the squares in $\D$ containing $x.$ One examines all those squares, starting at $Q_0$ and checks whether condition \eqref{st1} is satisfied. If it is not, then one proceeds to the square containing $x$ in the next generation.  The process stops when one finds a square $Q$ containing $x$ for which \eqref{st1} holds. Note that the set of $x$ for which the process never stops has vanishing $\mu$ measure by \eqref{eq2bis}. Hence $\sum_{Q\in \F_1} \mu(Q)=1.$
Since  $S_{Q_0}=0,$ it follows from \eqref{eq2b}  that it is necessary to descend at least $M+1$ generations to find the first stopping time square.

The second stopping time condition is slightly different. Let $Q\in \F_1$.  The second stopping time is performed on the relative martingale associated with $Q$ and its condition is
\begin{equation}\label{st2} 
\left| S_{Q,R} \right| > M\,b_M.
\end{equation}
A stopping time square $R$ of second generation satisfies $\left| S_{Q,R} \right| > M\,b_M$ and 
$$\left| S_{Q,R'} \right| \le M\,b_M, \quad R' \in \D, \quad R \subsetneq R' \subset Q.$$
By \eqref{eq2bis} and \eqref{eq3b} the stopping time squares of second generation cover almost all $Q.$ Again, by \eqref{eq7b} and the fact that $S_{Q,Q}=0$ one has to descend through at least $M+1$ generations to find a stopping time square of second generation. Hence if $R$ is a stopping time square of second generation and $R\in \D_n$ then $n \geq 2(M+1).$  We do not put all stopping time squares of second generation in $\F_2(Q).$ We put a stopping time square of second generation $R$ in $\F_2(Q)$ provided $S_{R} \in \sigma(S_Q, 120^\circ).$  That there are many such stopping time squares can be shown as follows.

Let $R$ be a stopping time square of second generation. Let $\a$ denote the angle between the vectors $S_{R} - S_Q$ and $S_{Q,R}$. Then by \eqref{eq3b},
$$
|S_{R} - S_Q| \ge |S_{Q,R}|- b_M \ge (M-1)\, b_M
$$
and
$$0\leq |\sin \a| \leq \frac{|S_{R} - S_Q - S_{Q,R}|}{|S_{R} - S_Q|} \leq \frac{b_M}{(M-1)\,b_M}=\frac{1}{M-1}<\sin 15^{\circ},$$
provided $M-1>1/\sin 15^{\circ},$ which we assume. Since $|S_{R} - S_Q - S_{Q,R}| < |S_{R} - S_Q|$ and $S_{R,Q}= S_{R} - S_Q+(S_{Q,R} -S_{R} - S_Q)$, we see that $\cos \alpha > 0.$  Thus $|\a|< 15^{\circ}.$  

By Lemma \ref{sector} there is $j$ with $1\le j\le 8$ such that $\tilde{\sigma}_j(S_Q) \subset \sigma(S_Q, 120^\circ).$ If we are lucky enough $S_{Q,R}  \in \sigma_j$ and so $S_R -S_Q \in  \tilde{\sigma_j},$ which yields $S_R \in \tilde{\sigma_j}(S_Q) \subset \sigma(S_Q, 120^\circ).$ 

But  it may occur that $S_{Q,R} \in \sigma_k, \; k\neq j.$   Applying two symmetries $f_{Q,j,k}$ of $Q,$ or a symmetry of the form $\gamma \circ f_{Q,j,k}$ in the worst case, as we discussed before Lemma \ref{martlemma}, we obtain a stopping time square $R'$  of second generation and of the same size as $R$ such that $S_{Q, R'}
\in \sigma_j$  and so $ S_{R'} \in \tilde{\sigma_j}(S(Q)) \subset \sigma(S_Q, 120^\circ),$  as desired.

Therefore, subdividing the stopping time squares of second generation in eight classes, according to the octant to which $S_{Q,R}$ belongs, we get
\begin{equation}\label{sumst2}
\sum_{R\in\F_2(Q)}\mu(R) \geq \frac{1}{8} \,\mu(Q).
\end{equation}
Define $\F_2=\cup_{Q\in \F_1} \F_2(Q).$

Let us obtain some properties of stopping time squares $R$ in $\F_2(Q).$ Let $\tilde R$ be the father of $R.$ Then $\left|S_{Q,\tilde R}\right|\leq M\,b_{M}$ and so 
$$\left| S_{\tilde{R}}-S_Q\right| \le  \left| S_{Q,\tilde{R}}\right|+\left| S_{\tilde{R}}-S_Q-S_{Q,\tilde{R}}\right| \le (M+1)b_M  $$
and
$$
|S_R-S_Q|\leq |S_R-S_{\tilde R}|+|S_{\tilde R}-S_Q|\leq b_{M}+(M+1)b_{M}=(M+2)b_{M}.
$$

Now two possibilities appear.

If $|S_Q|\leq 2|S_R-S_Q|\leq 2(M+2)b_{M}$, then
\begin{equation*}\label{}
|S_R|\leq |S_R-S_Q|+|S_Q|\leq 3(M+2)b_{M}.
\end{equation*}
If $|S_Q|> 2|S_R-S_Q|$,  since $S_R \in \sigma(S_Q,120^\circ)$ we can apply Lemma \ref{martlemma} to get
\begin{equation*}\label{}
|S_R|\leq |S_Q|-|S_R-S_Q|/4\leq |S_Q|- (M-1)b_{M}/4\leq |S_Q|- b_{M}
\end{equation*}
provided  $M\geq 5.$

Therefore at least one of the following two inequalities holds: either
\begin{equation}\label{mar10}
|S_R| \leq 3(M+2)b_{M},
\end{equation}
or
\begin{equation}\label{mar11}
|S_R| \leq |S_Q|- b_{M}.
\end{equation}

We can proceed to define inductively $\F_n$  for $n\geq3,$  in a way analogous to what we did to define $\F_2$ from $\F_1.$  Assume that we have defined $\F_{n-1}=\cup_{Q\in \F_{n-2}} \F_{n-1}(Q).$
Given $Q \in \F_{n-1}$ we set the $n$ generation stopping time in the relative martingale associated with $Q$ as
$$
|S_{Q,R}|> Mb_{(n-1)M} 
$$
If $R$ is a stopping time square of $n$-th generation then besides the previous inequality one has
$$
|S_{Q,R'}|\leq Mb_{(n-1)M}, \quad R' \in \D,	\quad R\varsubsetneq R' \subset Q, 
$$
whence
\begin{equation}\label{mar12}
|S_{R'}-S_Q| \leq |S_{Q,R'}| + b_{(n-1)M} \leq (M+1)b_{(n-1)M}.
\end{equation}
Note that if $R$ is a stopping time square of generation $n,$ we can take advantage of the symmetries of $Q,$ as before, to find another one, say $R',$
of the same size with the additional property that $S_{R'} \in \sigma(S_Q, 120^\circ).$ Define $\F_n(Q)$ as the stopping time squares $R$ of generation $n$ such that $S_R \in \sigma(S_Q, 120^\circ)$ and $\F_n=\cup_{Q\in \F_{n-1}} \F_{n}(Q).$
We then have
\begin{equation}\label{sumstn}
\sum_{R\in\F_n(Q)}\mu(R) \geq \frac{1}{8} \,\mu(Q).
\end{equation}
Given $R\in\F_{n}(Q)$, we have as before that at least one of the following two inequalities holds:  either
\begin{equation}\label{mar14}
|S_{R}|\leq 3(M+2)b_{(n-1)M}
\end{equation}
or
\begin{equation}\label{mar15}
|S_R|\leq |S_Q|-b_{(n-1)M}.
\end{equation}

Set $F=\bigcap_{n=1}^{\infty}\bigcup_{Q\in\F_n}Q.$
To complete the proof we shall show that the hypotheses of Hungerford's Lemma \ref{hungerford}  are fulfilled and that
\begin{equation}\label{mar17}
\lim_{m\to\infty}S_m(x)=0, \quad  x\in F.
\end{equation}

For (b) in Hungerford's Lemma \ref{hungerford} recall that each stopping time square has descended at least $M+1$ generations from the generating square in the previous family. Then one has (b) with $\e$ replaced by $\frac{1}{4^M}$ and taking $M$ big enough one has $\frac{1}{4^M} <\e.$  Condition (c) with $c=\frac{1}{8}$ is  \eqref{sumstn}. 

To prove \eqref{mar17}, take $x\in F$. For every $n=1,2,\dots$, there is a unique $Q_n \in\F_n$ such that $x\in Q_n.$  Let $m_n$ be the unique positive integer satisfying $Q_n \in\D_{m_n}.$   Clearly the sequence $m_n$ is increasing and $m_n > M\,n.$  Since $S_{Q_{n}}=S_{m_n}(x)$  we have by \eqref{mar14} and \eqref{mar15} that either
\begin{equation}\label{mar19}
\left|S_{m_n}(x)\right| \leq 3(M+2)b_{(n-1)M}
\end{equation}
or
\begin{equation}\label{mar18}
|S_{m_n}(x)|\leq |S_{m_{n-1}}(x)|- b_{(n-1)M}, \quad n=1, 2, \dots
\end{equation}

 For  $m_{n-1}< m<m_{n}$ we have by \eqref{mar12} 
\begin{equation}\label{mar20}
\left|S_{m}(x)-S_{m_{n-1}}(x)\right|\leq (M+1)b_{(n-1)M}.
\end{equation}
To conclude that $\lim_{m\to\infty}S_m(x)= 0$ it is enough to show that $\lim_{n\to\infty}S_{m_n}(x)= 0$.

We say that $n\in\mathcal N_1,$ if \eqref{mar18} holds and $n\in\mathcal N_2,$ if \eqref{mar19} holds and \eqref{mar18} fails. 
As $\sum_n b_n$ diverges and $(b_n)_{n=1}^\infty$ is non-increasing, also $\sum_n b_{(n-1)M}$ diverges. It follows that \eqref{mar18} cannot hold for infinitely many consecutive $n$, whence $\mathcal N_2$ is infinite.

Let $n\in\mathcal N_2$ and let $N>n$ be such that $k\in\mathcal N_1$ for all $n<k<N$. Then by \eqref{mar18} and \eqref{mar19} for $n<k<N$,
$$|S_{m_k}(x)|\leq |S_{m_n}(x)|\leq 3(M+2)b_{(n-1)M}.$$ 
It follows that $\lim_{m\to\infty}S_m(x)= 0$.

\end{proof}

\section{Appendix 1: a lemma on Hausdorff dimension}\label{A1}

\vspace{0,4cm}
Let $\mu$ be the canonical measure associated with a Cantor set in $\mathbb{R}^d$, as defined in the Introduction before the statement of Theorem \ref{pvthm2}. Denote by  $\D_n$ the set of all cubes $Q_j^n, \; 1\le j\le 2^{dn},$ appearing at the $n$-th generation of the construction and $\D=\cup_n \D_n.$

The following lemma is due to Hungerford, who worked in a one dimensional context; see \cite{H}. 

\begin{lm}\label{hungerford}
Let $0<\e<c<1$ and let $\F_n$ be a disjoint family of cubes in $\D$, for $n=0,1,2,\dots$, satisfying the following.
\begin{enumerate}
\item[(a)] $\F_0=\{Q_0\},$
\item[(b)] if $Q\in\F_{n+1}$, then there exists $\tilde Q\in\F_n$ with $Q\subset\tilde Q$ and $\mu(Q)\leq \e \mu(\tilde Q)$,
\item[(c)] if $Q\in\F_{n}$, then
$$\sum_{R\subset Q, \,R\in\F_{n+1}}\mu(R) \geq c\,\mu(Q).$$
\end{enumerate}
Let $E=\cap_n\cup_{Q\in\F_n}$. Then
$$\dim E \geq \alpha \left(1 - \log c/\log \e \right).$$
\end{lm}
\begin{proof}
Set $\beta= \alpha\left(1 - \log c/\log \e \right)$. We will construct a Borel probability measure $\nu$ with $\nu(E)=1$ such that for some constant $C$ and for all balls $B(x,r)$ centred at $x$ of radius $r$ one has
\begin{equation}\label{hunger}
\nu(B(x,r))\leq Cr^{\beta}, \quad \text{for}\; x\in E, \quad 0<r\leq 1.\end{equation} 
Then Frostman's lemma will give the result. 

Let us define the functions $\nu_n:\F_n\to \R, n=0,1,2\dots,$ setting first $\nu_0(Q_0)=1$. Suppose that $\nu_1,\dots,\nu_{n-1}$ are defined and let for $Q\in\F_n$, with $\tilde Q$ as in (b),
$$\nu_n(Q) = \frac{\nu_{n-1}(\tilde Q)}{\sum_{R\in \F_n, R\subset\tilde Q}\mu(R)}\mu(Q).$$
Then we define the Borel measures $\nu_n$ setting
$$\nu_n(A) = \sum_{Q\in \F_n}\frac{\nu_n(Q)}{\mu(Q)}\mu(A\cap Q)\ \text{for}\ A\subset \R^d.$$


Then for $Q\in \F_n$,
\begin{align*}
&\nu_{n+1}(Q)=\sum_{R\in \F_{n+1}, R\subset Q}\nu_{n+1}(R)=\\
&\sum_{R\in \F_{n+1}, R\subset Q}\frac{\nu_{n}(Q)}{\sum_{P\in \F_{n+1}, P\subset Q}\mu(P)}\mu(R)=\nu_{n}(Q).
\end{align*}
Iterating this we have 
\begin{equation}\label{hunger2}
\nu_{m}(Q)=\nu_{n}(Q)\ \text{for}\ Q\in \F_n, m>n.\end{equation}
In particular, each $\nu_n$ is a probability measure and some subsequence of $(\nu_n)$ converges weakly to a probability measure $\nu$ such that $\nu(Q)=\nu_n(Q)$ for $Q\in\D_n$. 


Since
$$\nu(\bigcup_{Q\in \F_n}Q)=\sum_{Q\in \F_n}\nu(Q)=\sum_{Q\in \F_n}\nu_n(Q)=1,$$
we have $\nu(E)=1$. Therefore $\nu(E\setminus\cup_{Q\in\F_n}Q)=0$ for every $n$, so  
\begin{equation}\label{hunger1}
\nu(Q)= \sum_{R\subset Q, R\in \F_{n+1}}\nu(R), \quad Q\in\F_n.
\end{equation}

It remains to verify \eqref{hunger}. 
First of all we have by condition (c) for $Q\in\F_n, n\geq 2$,
$$\frac{\nu(Q)}{\mu(Q)}=\frac{\nu_n(Q)}{\mu(Q)}=\frac{\nu_{n-1}(\tilde Q)}{\sum_{R \subset\tilde Q, R\in \F_n}\mu(R)}
\leq \frac{\nu(\tilde Q)}{c\mu(\tilde Q)},$$
and by induction,
\begin{equation}\label{hunger3}
\frac{\nu(Q)}{\mu(Q)}\leq c^{-n}\ \text{for}\ Q\in\F_n, n=1,2\dots.\end{equation}

Now let us prove that 
\begin{equation}\label{eq6}
\nu(Q) \leq Cd(Q)^{\beta}\ \text{for}\ Q\in\D.\end{equation} 
Take $n$ such that $\e^{n+1} \leq \mu(Q) < \e^n$. We may assume that $\nu(Q)>0$. Then $Q$ intersects a square $R$ in the family $\F_{n+1}.$  Since by (b) 
$\mu(R) \le \e^{n+1} \le \mu(Q)$, one has $R \subset Q.$ We have, by \eqref{hunger1} and \eqref{hunger3},
$$\nu(Q)= \sum_{R\subset Q, R\in \F_{n+1}}\nu(R)\leq c^{-n-1}\sum_{R\subset Q, R\in \F_{n+1}}\mu(R)\leq c^{-n-1}\mu(Q).$$
Since $\mu(Q)\le d(Q)^{\alpha}$ it is enough to show that $c^{-n}\mu(Q)\leq \mu(Q)^{\frac{\beta}{\alpha}}$ which is 
$$c^{-n}\leq \mu(Q)^{-\log c/\log \e},$$
that is, 
$$-n\log c\leq-(\log c/\log \e)\log\mu(Q),$$ or 
$n\leq\log \mu(Q)/\log \e$, which is a consequence of $\mu(Q) < \e^n$.

To finish, let $x\in E$ and $0<r\leq 1$. For some $n$,\ $x$ belongs to a square $Q\in\D_n$ with $d(Q)/4\leq r \leq d(Q)$.   Then $B(x,r)$ can meet at most $4^d$ squares of $\D_n$, and so by \eqref{eq6}, $\nu(B(x,r))\leq 4^d\,\nu(Q)\leq 4^d\,C\,d(Q)^{\beta}\leq 4^{\beta+d}\,C\,r^{\beta}$ and \eqref{hunger} follows. 
\end{proof}

\section{Appendix 2: the Riesz transforms in $\R^d.$}\label{A2}

We first slightly modify the argument in \cite{CPV} to show that $\sum_{n=1}^{\infty} a_n^2 =\infty$ yields divergence \text{a.e.} of the martingale. If the martingale converges in a set of positive measure, then also the principal values of the Riesz transform exist in a set $E$ of positive measure, by the analog of Lemma \ref{pv1}. By a result of Tolsa \cite[Theorem 8.13]{T} we find a set $F\subset E$ of positive measure on which the singular Riesz transform operator is bounded on $L^2(\mu_{|F}).$  In particular, the capacity of $F$ associated with the Riesz kernel is positive and so also that of the Cantor set. The main result of \cite{MT} (see Theorem 1.2, p. 678 and its extension in the last formula in p. 696)  states that the $\alpha$-Riesz capacity of the Cantor set is comparable to $(\sum_{n=1}^{\infty} a_n^2)^{-\frac{1}{2}},$ so that positive capacity yields a convergent series. We remark that the previous argument uses very strong results, in particular the non-homogeneous $T(1)$-Theorem of Nazarov, Treil and Volberg, to extract the subset $F$ on which the singular Riesz transform is $L^2(\mu_{|F})$ bounded. In \cite{CPV} one resorts to Menger curvature, which is not available for kernels of homogeneity $-\alpha$ with $1<\alpha<d,$ and the proof is slightly simpler. It would be desirable to have a direct argument relating the series to the convergence of the martingale.

The part of the stopping time argument of section \ref{st} that does not obviously extend  to higher dimensions is related to the sector
$\sigma(z,120^\circ).$ In particular, one should replace the $45^\circ$ degrees sectors centred at the origin with one edge on a coordinate axis with other regions. We proceed as follows. Divide $\R^d$ into $2^d$ regions (which in $\R^3$  are the usual octants) by requiring that each coordinate has a definite sign. For example,
$$
O=\{x \in \R^d: x_1 \ge 0, x_2\ge 0, \dots x_d \geq 0\}
$$
or
$$
O'=\{x \in \R^d: x_1 \le 0, x_2\ge 0, \dots x_d \geq 0\}
$$
are such regions. Let us concentrate in the region $O.$
Divide $O$ in the $ d !$ subregions determined by a permutation $\sigma$ of the $d$ variables
$$
O_\sigma= \{x \in \R^d:  0\le x_{\sigma(1)} \le x_{\sigma(2)} \le \dots \le x_{\sigma(d)}\}.
$$
Note that the maximal angle between two vectors lying in a subregion $O_{\sigma}$ is precisely $\arccos(d^{-1/2}),$ which approaches $90^\circ$ as $d \rightarrow \infty.$  Given a cone $\Gamma$ with vertex at the origin and aperture $\theta,$ we would like to find a region $O_\sigma$ contained in the cone $\Gamma.$ This can be done as follows. The axis of the cone is a ray emanating from the origin contained in $O_{\sigma}$ for some $\sigma.$ Taking $\theta =\theta(d)< \pi$ close enough to $\pi$ one can achieve $O_\sigma \subset \Gamma.$ Indeed, something stronger can be obtained: there exists a sufficiently small angle $\gamma=\gamma(d)$ such that expanding $O_{\sigma}$  in all directions by at most $\gamma$ degrees one still remains in the cone $\Gamma.$

The planar argument now works with $\theta$ in place of $120^\circ.$

One also needs to have enough linear isometries to transport one region $O_\sigma$ into another $O_{\sigma'}.$ Consider the following kinds of linear isometries. Fix a variable $x_i$ and take the mapping that leaves the other variables invariant and changes the sign to the $x_i$ variable. Given two variables $x_i$ and $x_j$ with $i\neq j$ consider the mapping that leaves the other variables invariant and interchanges $x_i$ and $x_j.$ Finally take the mapping $x \rightarrow -x.$ Let $\mathcal{S}$ the set of such linear isometries. One can easily check that given two regions $O_\sigma$ and $O_{\sigma'}$ one can map one into the other by composing finitely many isometries in $\mathcal{S}.$

All these elements lead to a stopping time that proves Theorem \ref{pvthm2}.

\bigskip

\noindent
\textbf{Acknowledgements.} The authors are grateful to X. Tolsa for various conversations on the subject and to the referee for many suggestions that have improved considerably the exposition.
Pertti Mattila is grateful for the invitation to visit the Department of Mathematics at Universitat Aut\`onoma de Barcelona.
J. Verdera acknowledges support from the grants 2021-SGR-00071(Generalitat de Catalunya), PID2020-112881GB-I00  and Severo Ochoa and Maria de Maeztu CEX2020-001084-M (Ministerio de Ciencia e Innovaci\'on). 
\medskip

\noindent
{\bf Statements and declarations.} The authors have no conflict of interests to declare.\medskip

\noindent
{\bf Data availability statement.} Data sharing not applicable to this article as no datasets were generated or analysed during the current study.

\vspace{1cm}

\begin{footnotesize}

{\sc Departament de Matem\`atiques,
Universitat Aut\`onoma de Barcelona, Catalonia,}\\
\emph{E-mail address:} 
\verb"julia.cufi@uab.cat" 

{\sc Departament de Matem\`atiques,
Universitat Aut\`onoma de Barcelona, Catalonia,}\\
\emph{E-mail address:} 
\verb"juan.jesus.donaire@uab.cat" 

{\sc Department of Mathematics and Statistics,
P.O. Box 68,  FI-00014 University of Helsinki, Finland,}\\
\emph{E-mail address:} 
\verb"pertti.mattila@helsinki.fi" 

{\sc Departament de Matem\`atiques,
Universitat Aut\`onoma de Barcelona and Centre de Recerca matem\`atica, Catalonia,}\\
\emph{E-mail address:} 
\verb"joan.verdera@uab.cat"

\end{footnotesize}

\end{document}